\documentclass[11pt]{article}

\usepackage{epsfig,amsmath,latexsym,amssymb}
\usepackage{amscd }
\usepackage{multirow}
\usepackage{graphicx}
\usepackage{enumitem}
\newcounter{thm}
\newcounter{ex}
\newcounter{re}
\newtheorem{Theorem}[thm]{Theorem}
\newtheorem{Lemma}[thm]{Lemma}

\newtheorem{Proposition}[thm]{Proposition}

\newtheorem{Corollary}[thm]{Corollary}
\newtheorem{Definition}[thm]{Definition}
\newenvironment{proof}[1][\unskip]{\addvspace{\baselineskip} \footnotesize \textbf{Proof #1.}}{ \hfill  $\square$ \par\addvspace{\baselineskip}}
\newenvironment{sketch}[1][\unskip]{\addvspace{\baselineskip}\footnotesize \textbf{Sketch #1.}}{ \hfill  $\square$ \par\addvspace{\baselineskip}}

\def\N{{\mathbb N}}  
\def\p{{\mathbb P}}  
\def\Z{{\mathbb Z}}

\def\E{{\mathbb E}}

\def\connected{\leftrightarrow}
\def\occupied{\Leftrightarrow}

\begin{document}
\title{Inhomogeneous Long-Range Percolation for Real-Life Network Modeling}

\author{Philippe Deprez\footnote{RiskLab, Department of Mathematics, 
ETH Zurich, 8092 Zurich, Switzerland} \qquad
Rajat Subhra Hazra\footnote{Indian Statistical Institute, Theoretical Statistics and Mathematics Unit, Kolkata 700 108, India} \qquad
Mario V.~W\"uthrich$^\ast$\footnote{Swiss Finance Institute SFI Professor, 8006 Zurich, Switzerland}}
\date{\today}
\maketitle

\begin{abstract}
The study of random graphs has become very
popular for real-life network modeling such as 
social networks or financial networks.
Inhomogeneous long-range percolation (or scale-free percolation) 
on the lattice $\mathbb Z^d$, $d\ge1$, is a particular
attractive example of a random graph model because
it fulfills several stylized facts of real-life networks.
For this model various geometric properties such as
the  percolation behavior, the degree distribution
and graph distances 
have been analyzed. In the present
paper we complement the picture about graph distances. Moreover, we prove continuity of the
percolation probability in the phase transition point.
\end{abstract}



\section{Introduction}

Random graph theory has become very popular to 
model real-life networks. Real-life networks may be understood as sets of particles 
that are possibly linked with each other. 
Such networks appear for example as virtual social networks,
see \cite{NSW}, or financial networks such as 
the banking system where banks exchange lines of credits 
with each other, see \cite{cont1} and \cite{Cont2}. 
Many different random graph models have been developed in recent years 
in order to understand the geometry of such networks. 
Using empirical data one has observed several
stylized facts about large real-life networks, for a detailed outline 
we refer to \cite{NSW} and Section 1.3 in \cite{Durrett}:
\begin{itemize}
\item
Distant particles are typically connected by very few 
links. This is called the ``small-world effect''. For example, there is 
the observation that most particles in real-life networks are connected  
by at most six links, see also \cite{Watts}.
\item
Linked particles tend to have common friends. This is called the 
``clustering property''.
\item
The degree distribution,  that is, the distribution of the number of links of a given particle,  
is heavy-tailed, i.e.~its survival probability has a power law decay. 
It is observed that in real-life networks the (power law) tail parameter $\tau$ 
is often between $1$ and $2$, for instance, for the movie actor network $\tau$
is estimated to be around $1.3$. For more explicit examples we refer to
\cite{Durrett}.
\end{itemize}
A well studied model in the literature is the homogeneous long-range percolation 
model on $\Z^d$, $d\ge1$. In this model, the particles are the vertices 
of $\Z^d$. Any two particles $x, y\in\Z^d$ are linked 
with probability $p_{xy}$ which behaves  as
$\lambda |x-y|^{-\alpha}$ for $|x-y|\to \infty$. 
This model has by definition of $p_{xy}$ a local clustering property. 
Moreover, if $\alpha \in(d,2d)$ the graph distance between 
$x, y\in\Z^d$, that is, 
the minimal number of links that connect $x$ and $y$, behaves roughly 
as $(\log|x-y|)^{1/\log_2(2d/\alpha)}$ as $|x-y|\to \infty$, see \cite{biskup}.
This behavior can be interpreted as a version of the small-world effect. 
But homogeneous long-range percolation does not fulfill the stylized fact of having a heavy-tailed 
degree distribution. Therefore, 
\cite{Remco} introduced the inhomogeneous long-range percolation 
model (also known as scale-free percolation model) on $\Z^d$ 
which extends the homogeneous model in such a way that 
the degree distribution turns out to be heavy-tailed. 
In the inhomogeneous long-range percolation model one assigns to each 
particle $x\in\Z^d$ a positive random weight 
$W_x$ whose distribution is heavy-tailed with 
tail parameter $\beta>0$. These weights make particles more or less attractive, i.e.~if
a given particle $x$ has a large weight $W_x$ then it plays the role of a hub
in the network. Given 
these weights, two particles $x, y\in\Z^d$ are then linked with 
probability $p_{xy}$ which is approximately 
$\lambda W_x W_y |x-y|^{-\alpha}$ for large $|x-y|$. 
For $\min\{\alpha,\beta\alpha\}>d$ the 
degree distribution  is heavy-tailed 
with tail parameter $\tau=\beta\alpha/d>1$, see Theorem 2.2 in \cite{Remco}. 
Hence, this model fulfills the stylized fact of having a 
heavy-tailed degree distribution. 
For real-life applications the interesting case is $\tau=\beta\alpha/d\in (1,2)$, 
and in this case, if in addition $\alpha>d$, 
the graph distance between $x,y \in\Z^d$ 
is of order $\log\log|x-y|$ as $|x-y|\to\infty$, see \cite{Remco} and Theorem 8 below. 
This is again a version of the small-world effect. 
One goal of this paper is to complement the 
picture about graph distances of \cite{Remco}
by providing analogous results to \cite{ Berger2, biskup, biskup11, trapman:2010} 
for inhomogeneous long-range percolation. 

~

In  homogeneous long-range percolation it is known 
that there is a critical constant $\lambda_c=\lambda_c(\alpha, d)$ such
that there is an infinite connected component of 
particles for $\lambda>\lambda_c$ and there is no
such component for $\lambda < \lambda_c$, i.e.~in the former case there
is an infinite connected network in $\Z^d$. 
This phase transition picture in homogeneous long-range percolation
can be traced back to the work of \cite{aizenman:newman, newman:schulman, schulman}. 
Later work concentrated more on the geometrical properties of percolation like graph  distances, see 
\cite{benjamini:berger, biskup, biskup11, CGS, trapman:2010}. 
A good overview of the literature for long-range percolation is provided in 
\cite{biskup, chadey}. 
For homogeneous long-range percolation it is known that for 
$\alpha\le d$ there is an infinite connected component for all $\lambda>0$, and
therefore $\lambda_c=0$. This infinite connected component contains all particles 
of $\Z^d$, a.s., i.e.~in
that case we have a completely connected network of all particles of $\Z^d$.
The case $\alpha\in(d,2d)$ is treated in \cite{Berger}. In
that case there is no infinite connected component at criticality $\lambda_c$.
This result combined with Proposition 1.3 of \cite{akn} 
shows continuity of the percolation probability, 
that is, the probability that a given particle belongs 
to an infinite connected component. For $\alpha\ge 2d$, the problem is still open, 
except in the case $d=1$ and $\alpha>2$ because in that latter case there 
does not exist an infinite connected component for any $\lambda>0$. 

In  inhomogeneous long-range percolation the conditions for the existence of a non-trivial 
critical value $\lambda_c \in (0,\infty)$  
were derived in \cite{Remco}, see also Theorems \ref{theo 31} and \ref{theo 32} below. 
The continuity of the percolation probability was conjectured in 
that article. One main goal of the present work is to prove this conjecture for $\alpha \in(d,2d)$. 
The crucial technique to prove this conjecture is the renormalization method presented in 
\cite{Berger}. This technique will also allow to complement the picture of graph
distances provided in \cite{Remco}.

~

\noindent
{\bf Organization of this article.}  In Section~\ref{sec:model}, we describe the model assumptions and notations. We also state the conditions that are required for a non-trivial phase transition. In Section~\ref{sec:results}, we state the main results of the article. Namely, we show the continuity of the percolation function in Theorem~\ref{theorem continuous} which is based on a finite box estimate stated in Theorem~\ref{continuity}. We also 
complement the picture about  graph distances of \cite{Remco}, see
Theorem \ref{chemical distance} below. In Section~\ref{sec:discussion}, we discuss open 
problems and compare the results to homogeneous
long-range percolation model results. Finally, we provide 
all proofs of our results in Section~\ref{sec:proofs}.

\section{Model assumptions and phase transition picture}\label{sec:model}

We define the inhomogeneous long-range percolation model of \cite{Remco}
in a slightly modified version. The reason for this modification is
that the model becomes easier to handle but it keeps the essential
features of inhomogeneous long-range percolation.
In particular, all  results of \cite{Remco} only depend on the asymptotic behavior
of survival probabilities. Therefore, we choose an explicit example which on the one
hand has the right asymptotic behavior and on the other hand is easy to handle. 
This, of course, does not harm the generality of the results.

Consider the lattice $\Z^d$ for fixed $d\ge 1$ with vertices $x\in \Z^d$
and edges $(x,y)$ for $x,y\in \Z^d$. Assume $(W_x)_{x\in \Z^d}$
are i.i.d.~Pareto distributed weights
with parameters $\theta=1$ and $\beta>0$, i.e.,~the weights $W_x$ have i.i.d.~survival probabilities
\begin{equation*}
\p \left[W_x > w\right] = w^{-\beta}, \qquad
\text{ for $w\ge 1$.}
\end{equation*}
Conditionally given these weights $(W_x)_{x\in \Z^d}$,
we assume that edges $(x,y)$ are independently from each other either
occupied or vacant. The conditional
probability of an occupied edge between
$x$ and $y$ is
chosen as
\begin{equation}\label{edge probability 2}
p_{xy} = 1 - \exp \left( -\frac{\lambda W_xW_y}{|x-y|^{\alpha}}
\right),\qquad \text{ for fixed given parameters $\alpha, \lambda \in (0,\infty)$}.
\end{equation}
For $|\cdot|$ we choose the Euclidean norm. 
If there is an occupied edge between $x$ and $y$ we write $x \occupied y$;
if there is a finite connected path of occupied edges between
$x$ and $y$ we write $x \connected y$ and we say
that $x$ and $y$ are connected. Clearly $\{ x \occupied y \}
\subset \{ x \connected y \}$. We define the cluster 
of $x\in \Z^d$ to be the connected component
\begin{equation*}
{\mathcal C}(x)=\{y \in \Z^d;~x\connected y\}.
\end{equation*}
Our aim is to study the size of the cluster ${\mathcal C}(x)$ and
to investigate its percolation properties
as a function of $\lambda>0$ and $\alpha>0$, 
that is, as a function of the edge probabilities
$(\lambda, \alpha) \mapsto
p_{xy}= p_{xy}(\lambda,\alpha)$. Note that ${\mathcal C}(x)$ exactly denotes
all particles $y\in \Z^d$ which can be reached within the network
whose links are described by the occupied edges.
The percolation probability
is defined by
\begin{equation*}
\theta(\lambda,\alpha) = \p \left[|{\mathcal C}(0)|=\infty\right].
\end{equation*}
This is non-decreasing in $\lambda$ and non-increasing in $\alpha$.
For given $\alpha>0$, the critical value $\lambda_c(\alpha)$ is
defined as
\begin{equation*}\lambda_c=
\lambda_c(\alpha) = \inf \left\{ \lambda>0;~\theta(\lambda,\alpha)>0\right\}.
\end{equation*}
Note that $\theta(\lambda,\alpha)$ and $\lambda_c(\alpha)$ 
also depend on $\beta$, but this parameter will be kept fixed. 

\noindent{\it Trivial case.} For $\min \{ \alpha, \beta \alpha \}\le d$,
we have $\lambda_c=0$.
This comes from the fact that for any $\lambda>0$
\begin{equation*}
\p\left[|\{y \in \Z^d;~0\occupied y\}|=\infty\right]=1,
\end{equation*}
see Theorem 2.1 in \cite{Remco}. This says that the degree distribution of a given
vertex is infinite, a.s., and therefore there is an infinite connected 
component, a.s. 
For this reason we only consider the non-trivial case
$\min \{ \alpha, \beta \alpha \}> d$ (where a phase transition may occur). 
In this latter case the degree distribution is heavy-tailed with tail 
parameter $\tau=\beta\alpha/d>1$, see Theorem 2.2 of \cite{Remco}. 

\begin{Theorem}[upper bounds] \label{theo 31}
Fix $d\ge 1$. Assume
$\min \{ \alpha, \beta \alpha \} > d$.
\begin{itemize}
\item[(a)] If $d\ge 2$, then $\lambda_c<\infty$.
\item[(b)]  If $d=1$ and $\alpha \in (1,2]$, then $\lambda_c<\infty$.
\item[(c)]  If $d=1$ and $\min \{ \alpha, \beta \alpha \} > 2$,
 then $\lambda_c=\infty$.
\end{itemize}
\end{Theorem}
\noindent
Since $W_x\ge 1$ , a.s., the edge probability stochastically dominates 
a configuration with independent edges being occupied with
probabilities $1- \exp(-\lambda |x-y|^{-\alpha})$. The latter is the
homogeneous long-range percolation model on $\Z^d$ and
it is well known that this model percolates 
(for $d\ge 2$ see \cite{Berger}; for $d=1$ and $\alpha\in(1, 2]$ see \cite{newman:schulman}).
For part (c) of the theorem we refer to Theorem 3.1 of \cite{Remco}. The next theorem follows
from Theorems 4.2 and 4.4 of \cite{Remco}.

\begin{Theorem}[lower bounds] \label{theo 32}
Fix $d\ge 1$. Assume
$\min \{ \alpha, \beta \alpha \} > d$.
\begin{itemize}
\item[(a)] If $\beta \alpha <2d$, then $\lambda_c = 0$.
\item[(b)] If $\beta \alpha >2d$, then $\lambda_c > 0$.
\end{itemize}
\end{Theorem}

\noindent
Theorems \ref{theo 31} and \ref{theo 32} give the phase
transition pictures for $d\ge1$, see Figure \ref{Picture: Phase} for an illustration. 
They differ for $d=1$ and $d\ge 2$ in that
the former has a region where $\lambda_c=\infty$ and the latter
does not. 
Note that $\beta\alpha<2d$ corresponds to  infinite variance of the degree distribution 
and $\beta\alpha>2d$ to finite variance of the degree distribution. In particular,
for the interesting case $\tau = \beta\alpha/d \in (1,2)$ we have $\lambda_c=0$, which implies that
for any $\lambda>0$ the network will have an infinite connected component, a.s.

\begin{figure}[htb!]
\begin{center}
\includegraphics[height=8cm]{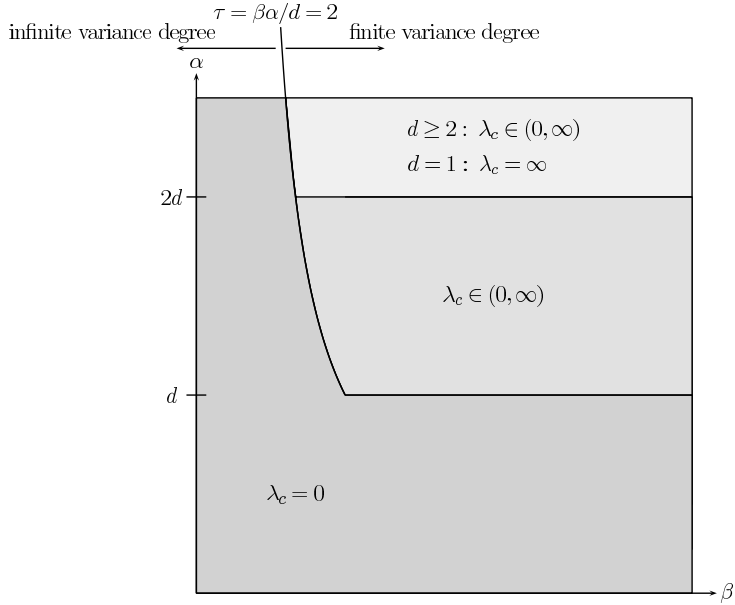}
\end{center}
\caption{phase transition picture for $d\ge1$.} \label{Picture: Phase}
\end{figure}

\section{Main results}\label{sec:results}
\subsection{Continuity of percolation probability}

We say that there exists an infinite cluster ${\mathcal C}$ if
there is an infinite connected component ${\mathcal C}(x)$ for some $x\in \Z^d$.
Since the model is translation invariant and ergodic, the event
of having an infinite cluster ${\mathcal C}$ is a
zero-one event. Thus, for $\lambda > \lambda_c$ there exists
an infinite cluster, a.s. Moreover, from Theorem 1.3 in
\cite{Berger} we know that an infinite cluster is unique, a.s.
This justifies the notation ${\mathcal C}$ for {\it the} infinite cluster
in the case of percolation $\theta(\lambda,\alpha)>0$ and implies that
we have a unique infinite connected network, a.s.
\begin{Theorem}\label{continuity}
Assume $\min \{ \alpha, \beta \alpha \} > d$ and 
$\alpha \in (d,2d)$. Choose $\lambda \in (0,\infty)$  with $\theta(\lambda,\alpha)>0$.
There exist $\lambda' \in (0,\lambda)$ and $\alpha'\in (\alpha, 2d)$ such that
\begin{equation*}
\theta\left(\lambda',\alpha'\right) > 0.
\end{equation*}
In particular, 
$\left\{ \lambda \in (0,\infty);~\theta(\lambda, \alpha)>0\right\}$ 
is an open interval in $(0,\infty)$, 
and there does not exist an infinite cluster ${\mathcal C}$
at criticality $\lambda_c$.
\end{Theorem}
\noindent 
Note that for $\beta\alpha<2d$ we have $\lambda_c=0$, therefore 
Theorems \ref{theo 32} and \ref{continuity} imply the following 
corollary. 
\begin{Corollary}\label{at criticality}
Assume $\alpha \in (d,2d)$ and $\tau=\beta\alpha/d>2$. 
There is no infinite cluster ${\mathcal C}$
at criticality $\lambda_c>0$.
\end{Corollary}
\noindent 
Next we state continuity
of the percolation probability in $\lambda$ which was conjectured in \cite{Remco}.
\begin{Theorem}\label{theorem continuous}
For $\min\{\alpha, \beta\alpha\}>d$ and $\alpha\in (d,2d)$, the percolation 
probability $\lambda \mapsto \theta(\lambda,\alpha)$ is continuous.
\end{Theorem}

\subsection{Percolation on finite boxes} 

For integers $m\ge 1$ we define the box of size $m^d$ by
$B_m = [0,m-1]^d$, and by
$C_m$ we denoted the largest connected component in box $B_m$ (with a fixed
deterministic rule if there is more than one largest connected component in $B_m$).

\begin{Theorem}\label{percolation:boxes}
Assume $\min\{\alpha,\beta\alpha\}>d$ and $\alpha \in (d,2d)$. 
Choose $\lambda\in (0,\infty)$ with $\theta(\lambda,\alpha)>0$. 
For each $\alpha'\in(\alpha, 2d)$ there exist 
$\rho>0$ and $N_0<\infty$ such that for all $m\ge N_0$ we have
\begin{equation*}
\p\left[|C_m|\ge \rho |B_m|\right] \ge 1-e^{-\rho m^{2d-\alpha'}}.
\end{equation*}
\end{Theorem}

\noindent This statement says that in case of percolation largest connected components
in finite boxes cover a positive fraction of these box sizes 
with high probability for large $m$, or in other words, the number of particles
belonging to the largest connected network in $B_m$ is proportional to $m^d$.
This is the analog to the
statement in homogeneous long-range percolation, see Theorem 3.2 in \cite{biskup}.

For integers $n\ge 1$ and $x\in \Z^d$ define the box centered at $x$ with total side length
$2n$ by $\Lambda_n(x)=x+[-n,n]^d$ and let ${\mathcal C}_n(x)$ be the vertices in $\Lambda_n(x)$
that are connected with $x$ within box $\Lambda_n(x)$. 
For $\ell < n$ and $\rho>0$ we denote by
\begin{equation*}
{\cal D}_n^{(\rho,\ell)}
=\left\{ x \in \Lambda_n(0); ~
|{\mathcal C}_\ell(x)|\ge \rho |\Lambda_\ell(x)| \right\}
\end{equation*}
the set of vertices $x\in \Lambda_n(0)$ which are $(\rho,\ell)$-dense, i.e., surrounded
by sufficiently many connected vertices in $\Lambda_\ell(x)$, 
see also Definition 2 in \cite{biskup}. 
\begin{Corollary} \label{corollary big clusters}
Under the assumptions of Theorem \ref{percolation:boxes} we have the following. 
\begin{enumerate}
\item[$(i)$] 
There exists $\rho>0$ such that for any $x\in \Z^d$
\begin{equation*}
\lim_{n\to\infty}
\p\left[|{\mathcal C}_n(x)| \ge \rho |\Lambda_n(x)| \big | x\in {\cal C}\right]=1.
\end{equation*}
%
\item[$(ii)$] 
For any $\alpha'\in (\alpha, 2d)$
there exist $\rho>0$ and $\ell_0$
such that for any $\ell$ and $n$ with $\ell_0\le \ell \le n/\ell_0$
\begin{equation*}
\p\left[|{\cal D}_n^{(\rho,\ell)}| \ge \rho |\Lambda_n(0)|\right]
\ge 1- e^{-\rho n^{2d-\alpha'}}.
\end{equation*}
\end{enumerate}
\end{Corollary}
\noindent
This result can be interpreted as local clustering in the sense that with high probability
(for large $n$) particles are surrounded by many other particles belonging to the 
same connected network. 
Corollary \ref{corollary big clusters} is the analog
to Corollaries 3.3 and 3.4 in \cite{biskup}. Once the proofs of Theorem
\ref{percolation:boxes} and Lemma \ref{lem:box} (a), below, are established it follows 
from the derivations in \cite{biskup}.

\subsection{Graph distances}

For $x,y\in \Z^d$ we define $d(x,y)$ to be the minimal number of occupied edges which connect
$x$ and $y$, and we set $d(x,y)=\infty$ for $y\notin {\cal C}(x)$.
The value $d(x,y)$ is called graph distance or chemical distance between $x$ and $y$, and
it denotes the minimal number of occupied edges that need to be crossed from $x$ to $y$
(and vice versa). If $d(x,y)$ is typically
small for distant $x$ and $y$ then we say that the network
has the small-world effect.
\begin{Theorem}\label{chemical distance}
Assume $\min \{ \alpha, \beta \alpha \} > d$.
\begin{itemize}
\item[(a)] (infinite variance of degree distribution). Assume  $\tau=\beta\alpha/d <2$. For
any $\lambda>\lambda_c=0$ there exists $\eta_1> 0$ such that
for every $\epsilon>0$
\begin{equation*}
\lim_{|x|\to \infty}
\p \left[ \left. \eta_1\le \frac{d(0,x)}{\log \log |x|}\le 
(1+\epsilon) \frac{2}{|\log(\beta\alpha/d-1)|}
 \right|0,x\in {\cal C}\right]=1.
\end{equation*}
\item[(b1)] (finite variance of degree distribution case 1). Assume  $\tau=\beta\alpha/d >2$
and $\alpha \in (d,2d)$. For any $\lambda>\lambda_c$  
and any $\epsilon>0$
\begin{equation*}
\lim_{|x|\to \infty}
\p \left[ \left. 
1-\varepsilon
\le \frac{\log d(0,x)}{\log\log |x|}\le 
(1+\epsilon)\frac{\log 2}{\log(2d/\alpha)} 
\right|0,x\in {\cal C}\right]=1.
\end{equation*}
\item[(b2)] (finite variance of degree distribution case 2). Assume 
 $\min\{\alpha,\beta\alpha\} >2d$. There exists $\eta_2>0$ such that
\begin{equation*}
\lim_{|x|\to \infty}
\p \left[\eta_2<  \frac{ d(0,x)}{ |x|} \right]=1.
\end{equation*}
\end{itemize}
\end{Theorem}
\noindent
From Theorem \ref{chemical distance} (a) we conclude that in the case $\tau \in (1,2)$ we have
a small-world effect and the graph distance is of order $\log \log |x|$ as $|x|\to \infty$.
In the case $\tau > 2$ and $\alpha\in (d,2d)$ (for $\lambda>\lambda_c$) the small-world effect
is less pronounced in that the graph distance is conjectured to be of order $(\log |x|)^\Delta$
for $|x|\to\infty$. Note that this is a conjecture because the bounds in 
Theorem \ref{chemical distance} (b1) are not sufficiently sharp to obtain
the exact constant $\Delta>0$. Finally, in the case 
$\min\{\alpha,\beta\alpha\} >2d$ we do not have the small-world effect
and graph distance behaves linearly in the Euclidean distance.
In Figure \ref{Picture: Distances} we illustrate Theorem \ref{chemical distance}
and we complete the conjectured picture 
about the graph distances. 

Case (a) of Theorem \ref{chemical distance} was proved in Theorems 5.1 and 5.3 of \cite{Remco}. 
Statement (b1) proves upper and lower bounds in case 1 of finite variance 
of the degree distribution. The
lower bound was proved in  Theorem 5.5 of \cite{Remco}.
The upper bound will be proved 
below in Proposition \ref{chemical distance proposition 1}. 
Finally, the lower bound in (b2) improves the one given in  Theorem 5.6 of \cite{Remco}.


\begin{figure}[htb!]
\begin{center}
\includegraphics[height=8cm]{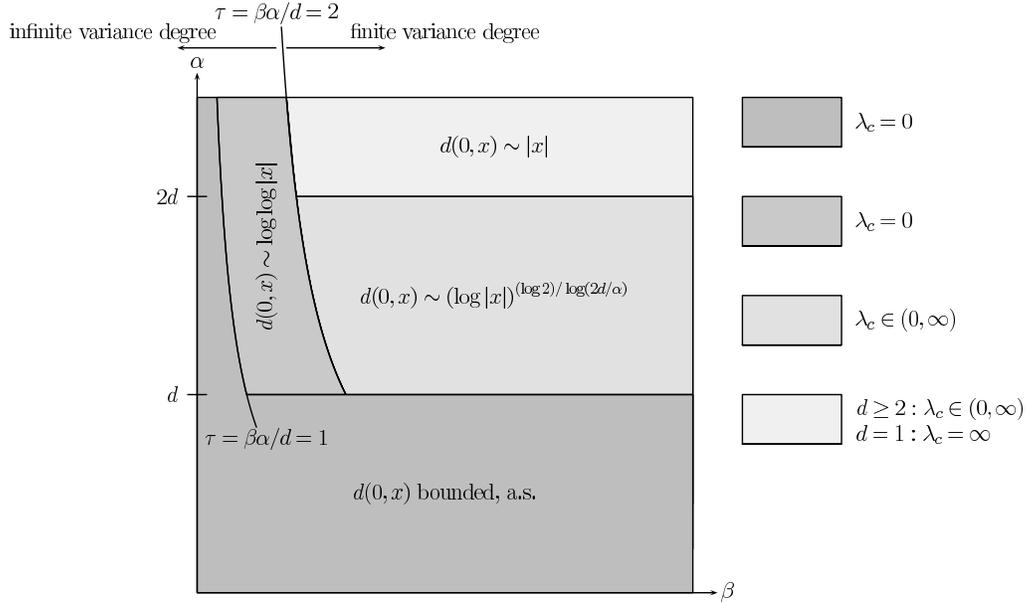}
  \end{center}
\caption{picture about the graph distances (partly as  conjecture).} \label{Picture: Distances}
\end{figure}

\section{Discussion and outlook}\label{sec:discussion}

In percolation theory one important problem is to understand the behavior of the model at criticality
$\lambda_c$. 
In nearest-neighbor Bernoulli bond percolation on $\Z^d$, where nearest-neighbor
edges are vacant or occupied with probability $p\in(0,1)$, 
it is known that for $d=2$ and for $d\ge 19$  there is no percolation at criticality and hence the 
percolation function is continuous at the critical value (see \cite{hara:slade} and \cite{grimmett} for more details). 
In cases $3\le d\le 18$ this question is still open. In the homogeneous long-range percolation model 
it was shown by \cite{Berger} that there is no percolation at criticality for $\alpha\in(d,2d)$. It is 
believed that the long-range percolation model behaves similarly to the nearest-neighbor Bernoulli percolation model 
when $\alpha>2d$ and, thus, showing continuity for such values 
and $d>1$ remains a difficult problem. 
In our model also the case $\min\{\alpha,\beta\alpha\}>2d$ for $d>1$ 
is open which is conjectured to behave as nearest-neighbor Bernoulli percolation, and
hence is not of interest for real-life network modeling.

Another problem which remains to be answered in both homogeneous and inhomogeneous long-range models is the 
continuity of the critical parameter $\lambda_c(\alpha)$ as a function of $\alpha$ and also as a 
function of parameter $\beta$, the exponent of the power law in weights (in case of the inhomogeneous model). Moreover, for real-life network applications it will be important to (at least) get reasonable bounds
on the critical value $\lambda_c(\alpha)$ and the percolation probability $\theta(\lambda,\alpha)$.
This will allow for model calibration of real-life networks so that (asymptotic) network properties
can be studied.

There was quite 
some work done to understand the geometry of the 
homogeneous long-range percolation model. 
In particular, there are five different behaviors depending on 
$\alpha<d, \alpha=d, \alpha\in (d,2d), \alpha=2d$ and $\alpha>2d$, for a review of existing results see 
discussion in~\cite{chadey}. 
In some of the cases, like $\alpha=2d$ (for $d\ge 1$) and $\alpha>2d$, the results are not yet fully known.
 The case $d=1$ and $\alpha=2$ was resolved recently in \cite{ding}. It is clear that in the case of inhomogeneous 
long-range percolation the complexity even increases due to having more parameters and, hence, degrees of freedom.
For instance, the understanding of the graph distance behavior is still poor for $\min\{\alpha,\beta\alpha\}>2d$, though
we believe that it should behave similarly to nearest-neighbor Bernoulli bond percolation. 
Moreover, the optimal constants in the asymptotic behaviors of 
Theorem \ref{chemical distance} are still open. However, we would like to emphasize that the 
inhomogeneous
long-range percolation model fulfills the stylized fact that the degree distribution is heavy-tailed 
which is not the case for the homogeneous long-range percolation model.
Therefore, the inhomogeneous model is an appealing framework for real-life network modeling,
in particular for $\tau \in (1,2)$ where we obtain
an infinite connected network for any $\lambda>0$.

\section{Proofs}\label{sec:proofs}
\subsection{Bounds on percolation on finite boxes}\label{renormalization section}

The basis for all the proofs of the previous statements is
Lemma \ref{connected component} below 
which determines large connected components on finite boxes. 
\begin{Lemma}\label{connected component}
Assume $\min \{ \alpha, \beta \alpha \} > d$ and $\alpha \in (d,2d)$.
Choose $\lambda \in (0,\infty)$ with $\theta(\lambda,\alpha)>0$ and 
let $\alpha'\in[\alpha, 2d)$.
For every $\varepsilon \in (0,1)$ and $\rho>0$ there exists $N_0\ge 1$
such that for all $m\ge N_0$
\begin{equation*}
\p\left[|C_m| \ge \rho m^{\alpha'/2}\right]\ge 1-\varepsilon,
\end{equation*}
where  $C_m$ is the largest connected component in box $B_m=[0,m-1]^d$.
\end{Lemma}

\begin{sketch}[of proof of Lemma \ref{connected component}]
This lemma corresponds to Lemma 2.3 of \cite{Berger} in our model. 
Its proof is based on renormalization arguments which only depend 
on the fact that $\alpha\in(d,2d)$ and that the edge 
probabilities are bounded from below by 
$1-\exp(-\lambda|x-y|^{-\alpha})$ for any $x,y\in\Z^d$. 
Using that $W_x\ge1$ for all $x\in\Z^d$, a.s., we see by  
stochastic dominance that the renormalization holds also 
true for our model. 
Renormalization shows that for $m$ sufficiently large, the probability of 
$\{$$B_m$ contains at least a positive fraction of $m^d$  
vertices that are connected within a fixed enlargement of $B_m$$\}$   
is bounded by a multiple of the probability of the same 
event but on a much smaller scale. 
To bound the latter probability we then use the fact that 
the model is percolating, and from this we can conclude Lemma \ref{connected component}. 
We  skip the details of the proof of Lemma  \ref{connected component} and refer to 
the proof of Lemma 2.3 of \cite{Berger11} for the details, 
in particular, the bound on $\psi_n$ in our homogeneous percolation model 
(see proof of Lemma 2.3 in \cite{Berger11}) also applies to the 
inhomogeneous percolation model.
\end{sketch}

Although the above lemma does not allow the connected component $C_m$ 
to have size proportional to the size of  box $B_m$, it is useful because
it allows to start a new renormalization scheme to improve
these  bounds. This results in our Theorem \ref{percolation:boxes}
and is done similar as in Section 3 of \cite{biskup}. 
For the proof of Theorem \ref{percolation:boxes}
we use the following lemma which has two parts. 
The first one gives the initial step of the renormalization 
and the second one gives a standard site-bond percolation model result. 
Once the lemma is established the proof of Theorem \ref{percolation:boxes}
becomes a routine task.

Let $C_m(x)$ denote the largest connected component in box $B_m(x)$ 
(with a fixed
deterministic rule if there is more than one largest connected component in $B_m(x)$).
For $x, y\in m\Z^d$, we say that boxes $B_m(x)$ and $B_m(y)$
are \textit{pairwise attached},  
write $B_m(x)\occupied B_m(y)$,
if there is an occupied edge between a vertex in 
$C_m(x)$ and a vertex in $C_m(y)$.

\begin{Lemma}\label{lem:box}~
 \begin{itemize}
  \item[(a)] Assume $\min\{\alpha,\beta\alpha\}>d$ and $\alpha\in (d,2d)$. 
Choose $\lambda\in (0,\infty)$ such that $\theta(\lambda,\alpha)>0$.
For each $\xi<\infty$ and $r \in (0,1)$ 
there exist  $m<\infty$ and an integer $\delta>0$ such that
\begin{eqnarray*}
&&\p\left[ |C_m(x)|<\delta |B_m(x)|\right] \le 1-r,
\\&&
\p\left.\bigg[ 
B_m(x) \occupied B_m(y)
\bigg||C_m(x)|\ge\delta |B_m(x)|,~|C_m(y)|\ge\delta |B_m(y)|
\right]\ge 1- e^{-\xi\left(\frac{|x-y|}{m}\right)^{-\alpha}},
\end{eqnarray*}
for all $x \neq y\in m\Z^d$.
\item[(b)] [Lemma 3.6, \cite{biskup}] Let $d\ge 1$ and 
consider the site-bond percolation model on $\Z^d$ with sites being 
alive with probability $r\in[0,1]$
and sites $x,y\in\Z^d$ are attached with probability 
$\widetilde{p}_{x,y}=1-\exp(-\xi |x-y|^{-\alpha})$ where
$\alpha\in (d,2d)$ and $\xi\ge 0$. Let 
$|\widetilde{\mathcal C}_N|$ be the size of the 
largest attached cluster $\widetilde{\mathcal C}_N$ of living sites 
in box $B_N$. For each $\alpha'\in(\alpha,2d)$ there 
exist $N_0\ge1$, $\nu>0$ and $\xi_0<\infty$ such that 
\begin{equation*}
\p_{\xi, r}\left[|\widetilde{\mathcal C}_N|\ge \nu |B_N|\right]
\ge 1- e^{-\nu \xi N^{2d-\alpha'}}
\end{equation*}
holds for all $N\ge N_0$ whenever $\xi\ge \xi_0$ and $r\ge 1- e^{-\nu \xi}$.
 \end{itemize}
\end{Lemma}

\begin{proof}[of Lemma \ref{lem:box} (a)]
We adapt the proof of Lemma 3.5 of \cite{biskup} to our model.
Fix $r\in(0,1)$ and $\xi<\infty$.  
Choose $\rho>0$ such that 
\begin{equation*}
{\lambda }{\left(2\sqrt{d}+1\right)^{-\alpha}}\rho^2 = \xi,
\end{equation*}
note that this differs from choice (3.13) in \cite{biskup}.
Lemma \ref{connected component} then provides that
there exists $N_0\ge 1$ such that
for all $m\ge N_0$
\begin{equation*}
\p\left[ |C_m|< \rho m^{\alpha/2}\right]\le 1-r.
\end{equation*}
For the choice 
$\delta=\rho m^{\alpha/2-d}$ the first part of the result follows.
For the second part we choose $x\neq y \in m\Z^d$. 
For $x'\in B_m(x)$ and $y'\in B_m(y)$ we have upper bound, 
using that $W_z\ge 1$ for all $z\in \Z^d$, a.s., 
\begin{equation}\label{bound to be used again}
1-p_{x'y'}
~\le~ \exp \left( -\lambda|x'-y'|^{-\alpha} \right)
~\le~ \exp \left( {-\lambda}\left(2\sqrt{d}+1\right)^{-\alpha} {|x-y|^{-\alpha}} \right),
\end{equation}
a.s., where the latter no longer depends on the weights
$(W_z)_{z\in \Z^d}$.
For our choices of $\delta$ and $\rho$, \eqref{bound to be used again} implies
\begin{eqnarray*}
&&\hspace{-1cm}
\p\left.\bigg[ 
B_m(x) \not\occupied B_m(y)
\bigg||C_m(x)|\ge\delta |B_m(x)|,\,|C_m(y)|\ge\delta |B_m(y)|
\right]
\\&&=
\E\left[\left. \prod_{x'\in C_m{(x)},y'\in C_m{(y)}} (1-p_{x'y'})\right|
|C_m(x)|\ge\delta |B_m(x)|,~|C_m(y)|\ge\delta |B_m(y)|\right]
\\&&
\le \exp\left(-\lambda \left(2\sqrt{d}+1\right)^{-\alpha}
|x-y|^{-\alpha}
\rho^2m^{\alpha}\right)
= \exp\left(-\xi
\left(\frac{|x-y|}{m}\right)^{-\alpha}\right).
\end{eqnarray*}
This shows the second inequality of part (a). For part (b) we refer to Lemma 3.6 in \cite{biskup}.
\end{proof}

\begin{proof}[of Theorem \ref{percolation:boxes}]
The proof follows as in Theorem 3.2 of \cite{biskup}, we briefly sketch the main argument.
Choose the constants $N_0\ge 1$, 
$\nu>0$, $\xi>\xi_0$, $r\ge 1-e^{-\nu\xi}$
and $\delta>0$ as in Lemma~\ref{lem:box},
and note that it is sufficient to prove the theorem for $L=mN$, where $N\ge N_0$ and $m$ is 
chosen (fixed) as in
Lemma~\ref{lem:box} (a).
In this set up $B_L$ can be viewed as a disjoint 
union of $B_m(x)$ for $x\in (m\Z^d\cap B_L)$. There are $N^d$ such disjoint boxes. We call
$B_m(x)$ alive if $|C_m{(x)}|\ge\delta |B_m|$ and we say that disjoint $B_m(x)$ and $B_m(y)$ 
are pairwise attached if their largest connected components 
$C_m{(x)}$ and $C_m{(y)}$ share an occupied edge.  
Part (a) of Lemma \ref{lem:box} provides that
$B_m(x)$ is alive with probability exceeding $r$ 
and $B_m(x)$ and $B_m(y)$ are pairwise attached with probability exceeding
$\widetilde{p}_{x,y}$ for living boxes $B_m(x)$ and $B_m(y)$ with
$x,y \in m\Z^d$ (note that in this site-bond percolation model the attachedness property is 
only considered between living vertices because these form the clusters).  
For any $N\ge N_0$,
let $A_{N,m}$ be the event that  box $B_L$ contains a connected component formed by
attaching at least $\nu|B_N|$ of the living boxes. 
On  event $A_{N,m}$ we have for the largest connected component
in $B_L$
\begin{equation*}
|C_L|\ge (\nu|B_N|)(\delta |B_m|)= \nu\delta |B_L|,
\end{equation*}
thus, the volume of the largest connected component $C_L$ in box $B_L$ is
proportional to the volume of that box and there remains to show that
this occurs with sufficiently large probability. 
Part (b) of Lemma \ref{lem:box} and stochastic dominance provides
(note that we scale $x,y\in m\Z^d$ from Lemma \ref{lem:box} (a) to the
site-bond percolation model on $\Z^d$ in Lemma \ref{lem:box} (b))
\begin{eqnarray*}
\p \left[|C_L|\ge \nu\delta |B_L| \right]&\ge&
\p\left[A_{N,m}\right]
~\ge~ \p_{\xi,r}\left[|\widetilde{\mathcal C}_N|\ge \nu |B_N|
\right]
\\&\ge& 1-e^{-\nu\xi N^{2d-\alpha'}}
~=~1-e^{-\nu\xi m^{\alpha'-2d} L^{2d-\alpha'}}.
\end{eqnarray*}
Choosing $\rho\le \min\{\nu\delta, \nu \xi m^{\alpha'-2d}\}$ provides
$$\p\left[ |C_L|\ge \rho |B_L|\right] \ge1- e^{-\rho L^{2d-\alpha'}}.$$
This finishes the proof of Theorem \ref{percolation:boxes}.
\end{proof}

\begin{proof}[of Corollary \ref{corollary big clusters}]
The proofs of $(i)$ and $(ii)$ of Corollary \ref{corollary big clusters} follow completely analogous 
to the proofs of Corollaries 3.3 and 3.4 in \cite{biskup}
(note that Lemma \ref{lem:box} (a) replaces Lemma 3.5 of \cite{biskup}
and Theorem \ref{percolation:boxes} replaces Theorem 3.2 of \cite{biskup}).
\end{proof}

\subsection{Proof of continuity of the percolation probability}

The key to the proofs of the continuity statements is
again Lemma \ref{connected component}.

\begin{proof}[of Theorem \ref{continuity}]
Note that $\min \{ \alpha, \beta \alpha \} > d$ and $\alpha \in (d,2d)$ imply 
that $\lambda_c<\infty$.
Therefore, there exists $\lambda \in (\lambda_c,\infty)$
with $\theta=\theta(\lambda,\alpha)>0$.
For these choices of $\lambda>0$
we have a unique infinite cluster ${\mathcal C}$,
a.s., and we can apply Lemma \ref{connected component}.

We consider the same site-bond percolation model on $\Z^d$ as in 
Lemma \ref{lem:box} (b). Choose $\alpha'\in (\alpha, 2d)$, 
$0<\chi<1-\varepsilon<1$
and $\kappa >0$ and define the model as follows:
the following events are independent and
every site $x\in \Z^d$ is alive with probability
$r=1-\varepsilon-\chi \in (0,1)$ and sites $x,y\in \Z^d$ are attached
with probability $\widetilde{p}_{xy}=1-\exp(-\kappa(1-\chi)|x-y|^{-\alpha'})$.
For given $\alpha'\in (\alpha, 2d)$
we choose the parameters $\varepsilon,\chi, \kappa$ such that
there exists an infinite attached cluster of
living vertices, a.s., which is possible (see proof of Theorem 2.5
in \cite{Berger}).

The proof is now similar to the one of  Theorem \ref{percolation:boxes}.
Choose 
$\rho>0$ such that $\lambda \left(2\sqrt{d}+1\right)^{-\alpha'}\!\!\rho^2= \kappa$.
From Lemma \ref{connected component} we know that for all
$m$ sufficiently large and any $x\in m\Z^d$ 
\begin{equation*}
\p\left[ \left|C_m(x)\right|\ge \rho m^{\alpha'/2}\right]\ge 1-\varepsilon
>1-\varepsilon-\chi=r,
\end{equation*}
where $C_m(x)$ denotes the largest connected component in $B_m(x)$. 
The latter events define alive vertices $x$ on the lattice $m\Z^d$
(which due to scaling is equivalent to the above aliveness 
in the site-bond percolation model on $\Z^d$).
Note that this aliveness property is independent between different
vertices $x\in m\Z^d$. Attachedness $B_m(x) \occupied B_m(y)$, for $x\neq y\in m\Z^d$,
is then used as in the proof of Theorem \ref{percolation:boxes} and 
we obtain in complete analogy to the proof of the latter theorem
\begin{eqnarray}\nonumber
&&\hspace{-1cm}
\p \left[B_m(x)\Leftrightarrow B_m(y) \Big| |C_m(x)|\ge \rho m^{\alpha'/2},\, |C_m(y)|\ge \rho m^{\alpha'/2}\right]
\\&&
\ge 
1 - \exp \left( {-\lambda}\left(2\sqrt{d}+1\right)^{-\alpha}
|x-y|^{-\alpha} \rho^2m^{\alpha'}
\right)\label{lower bound for alive and attached}
~\ge~
1 - \exp \left( -\kappa {\left(\frac{|x-y|}{m}\right)^{-\alpha'}} \right)
,\nonumber
\end{eqnarray}
where in the last step we used the choice of $\rho$ and the fact that 
$\alpha<\alpha'$. 
Since $\kappa > \kappa (1-\chi)$
we get percolation and there exists an infinite
cluster ${\mathcal C}$, a.s., which implies $\theta(\lambda,\alpha)>0$.
Of course, this is no surprise because of the choice $\lambda>\lambda_c$ with
 $\theta(\lambda,\alpha)>0$.

Note that the probability
of a vertex $x\in m\Z^d$ being alive depends only on finitely many
edges of maximal distance $\sqrt{d}m$
(they all lie in the box $B_m(x)$)
 and therefore this probability is a continuous function
of $\lambda$ and $\alpha$. This implies that 
we can choose $\delta \in (0, \chi\lambda)$ and $\gamma \in (0, \alpha'-\alpha)$ so small
that
\begin{equation*}
\p_{\lambda-\delta, \alpha+\gamma}
\left[ \left|C_m(x)\right|\ge \rho m^{\alpha'/2}\right]\ge 1-\varepsilon-\chi=r,
\end{equation*}
where $\p_{\lambda-\delta, \alpha+\gamma}$ is the measure where for occupied
edges we replace parameters $\lambda$ by $\lambda-\delta \in (0,\lambda)$ and 
$\alpha$ by $\alpha+\gamma \in (\alpha, \alpha')$.
As in \eqref{lower bound for alive and attached} we obtain, 
note $\alpha+\gamma <\alpha'$, 
\begin{eqnarray*}
&&\hspace{-1cm}
\p_{\lambda-\delta,\alpha+\gamma} \left[
B_m(x)\Leftrightarrow B_m(y) \Big| |C_m(x)|\ge \rho m^{\alpha'/2},\, |C_m(y)|\ge \rho m^{\alpha'/2}
\right]
\\&&
\ge 
1 - \exp \left( -(\lambda-\delta)\left(2\sqrt{d}+1\right)^{-(\alpha+\gamma)}
|x-y|^{-(\alpha+\gamma)} \rho^2m^{\alpha'}
\right)\nonumber
\\&&\ge
1 - \exp \left( -\kappa 
 \left(1-\delta/\lambda\right)
{\left(\frac{|x-y|}{m}\right)^{-\alpha'}} \right)
.\nonumber
\end{eqnarray*}
Since $\delta/\lambda<\chi$ we get percolation
and there exists an infinite
cluster ${\mathcal C}$, a.s., 
which implies that $\theta(\lambda-\delta,\alpha+\gamma)>0$.
This finishes the proof of Theorem \ref{continuity}.
\end{proof}

\begin{proof}[of Theorem \ref{theorem continuous}]
We need to modify  Proposition 1.3 of \cite{akn} because in our
model edges are not occupied independently induced by the random choices of
weights $(W_x)_{x\in \Z^d}$.

\noindent {\it (i)} From Theorem \ref{continuity} it follows 
that $\theta(\lambda,\alpha)=0$ for all $\lambda\in (0,\lambda_c]$,
which proves continuity of $\lambda \mapsto 
\theta(\lambda,\alpha)$ on $(0,\lambda_c]$.

\noindent {\it (ii)} Next we  show that
$\lambda \mapsto \theta(\lambda,\alpha)$  is left-continuous
on $\lambda> \lambda_c$, that is,
\begin{equation} \label{eq:leftcont}
 \lim_{\lambda'\uparrow \lambda} \theta(\lambda',\alpha)=\theta(\lambda,\alpha).
\end{equation}
To prove this we couple all percolation realization as $\lambda$ varies.
This is achieved by randomizing the percolation constant
$\lambda$, see \cite{akn} and \cite{BK}.
Conditionally given the i.i.d.~weights $(W_x)_{x\in \mathbb Z^d}$,
define a collection of independent exponentially distributed random variables $\phi_{(x,y)}$, 
indexed by the edges $(x,y)$, which have conditional distribution
\begin{equation}
\label{exponential distribution}
\mathbf{P}\left[\left. \phi_{(x,y)}\le \ell\right| (W_x)_{x\in \mathbb Z^d}\right] 
=1-\exp\left( - \frac{\ell W_x W_y}{|x-y|^\alpha}\right), 
\qquad \ell \in(0, \infty),
\end{equation}
compare to \eqref{edge probability 2}. We denote the probability
measure of $(\phi_{(x,y)})_{x,y\in \Z^d}$ by $\mathbf{P}$ in order to 
distinguish this coupling model.
We say that an edge $(x,y)$ is $\ell$-open if $\phi_{(x,y)}< \ell$,
and we define the connected cluster ${C}_\ell(0)$ of the origin to be the
set of all vertices $x\in \Z^d$ which are connected to the origin by an
$\ell$-open path. Note that we have a natural ordering in $\ell$, 
i.e.~for $\ell_1<\ell_2$ we obtain ${C}_{\ell_1}(0)
\subset {C}_{\ell_2}(0)$. Moreover for $\ell=\lambda>0$, 
the $\lambda$-open edges
are exactly the occupied edges in this coupling
(note that the exponential distribution
\eqref{exponential distribution} is absolutely continuous). This implies
 for $\ell=\lambda$
\begin{equation*}
\theta(\lambda,\alpha)
=\p \left[|{\mathcal C}(0)|=\infty\right]
=\mathbf{P} \left[ |C_{\lambda}(0)|=\infty\right].
\end{equation*}
By countable subadditivity of $\mathbf{P}$ 
and the increasing property of $C_\ell(0)$ in $\ell$
we have
\begin{equation*}
\lim_{\lambda'\uparrow \lambda} \theta(\lambda',\alpha)
= \mathbf{P}\left[ | C_{\lambda'}(0)|=\infty\text{ for some } 
\lambda'<\lambda\right].
\end{equation*}
Moreover, the increasing property of $C_\ell(0)$ in $\ell$ provides
 $\{|C_{\lambda'}(0)|=\infty \text{ for some }\lambda'<\lambda\}
\subset \{|C_{\lambda}(0)|=\infty\}$.
Therefore, to prove \eqref{eq:leftcont} it suffices to show that
\begin{equation*}
\mathbf{P}\left[\{|C_{\lambda'}(0)|<\infty \text{ for all } 
\lambda'<\lambda\} \cap\{ |C_{\lambda}(0)|=\infty\} \right]=0.
\end{equation*}
Choose $\lambda_0 \in (\lambda_c, \lambda)$. 
Since there is a unique infinite cluster for $\lambda_0>\lambda_c$, a.s.,
there exists an infinite cluster $C_{\lambda_0}\subset C_{\lambda}(0)$
on the set $\{ |C_{\lambda}(0)|=\infty\}$.
If the origin belongs to $C_{\lambda_0}$ then the proof is done. 
Otherwise, because $C_{\lambda_0}$ is a subgraph of $C_{\lambda}(0)$,
there exists a finite path $\pi$ of $\lambda$-open edges connecting 
the origin with an edge in $C_{\lambda_0}$. 
By the definition of $\lambda$-open edges we have
$\phi_{(x,y)}<\lambda$ for all edges $(x,y)\in \pi$. 
Since $\pi$ is finite we obtain the
strict inequality $\lambda_1 =\max_{(x,y)\in \pi} \phi_{(x,y)}<\lambda$.
Choose $\lambda'\in (\lambda_0\vee \lambda_1, \lambda)$ and it follows that
$|C_{\lambda'}(0)|=\infty$. This completes the proof for 
the left-continuity in $\lambda$.

\noindent {\it (iii)}
Finally, we need to prove right-continuity of $\lambda \mapsto \theta(\lambda, \alpha)$
on $\lambda \ge \lambda_c$.
For integers $n>1$ we consider the boxes $\Lambda_n=[-n,n]^d$ centered
at the origin, see also Corollary \ref{corollary big clusters}.
We define the events $A_n=\{{\cal C}(0) \cap \Lambda_n^{\rm c}\not=\emptyset \}$, i.e., the
connected component ${\cal C}(0)$ of the origin leaves the box $\Lambda_n=[-n,n]^d$.
Note that $\theta(\lambda, \alpha)$ is the decreasing limit of $\p[A_n]$ as
$n\to \infty$. Therefore, it suffices to show that $\p[A_n]$
is a continuous function in $\lambda$. We write $\p_\lambda=\p$ to 
indicate on which parameter $\lambda$ the probability law depends.
We again denote by ${\mathcal C}_n(0)$ the connected component of the origin
connected within box $\Lambda_n$, see Corollary \ref{corollary big clusters}.
Then, we have
\begin{equation*}
A_n=\{{\mathcal C}(0) \cap \Lambda_n^{\rm c}\not=\emptyset \}
=\{{\mathcal C}_n(0)\occupied \Lambda_n^{\rm c} \}.
\end{equation*}
Choose $\delta_0 \in (0,\lambda)$, then we have for all $\lambda'\in (\lambda-\delta_0,\lambda+\delta_0)$
and all $n'>n$
\begin{eqnarray}\nonumber
\left|\p_ \lambda\left[A_n\right]-\p_{\lambda'}\left[A_n\right]\right|
&=&\left|\p_ \lambda\left[
{\mathcal C}_n(0) \occupied \Lambda_n^{\rm c}
\right]-\p_{\lambda'}\left[{\mathcal C}_n(0)\occupied \Lambda_n^{\rm c}\right]\right|
\\&\le& \nonumber
\left|\p_ \lambda\left[{\mathcal C}_n(0)\occupied (\Lambda_n^{\rm c} \cap \Lambda_{n'})\right]-\p_{\lambda'}
\left[{\mathcal C}_n(0)\occupied (\Lambda_n^{\rm c}\cap \Lambda_{n'})\right]\right|
\\&&\hspace{.5cm}\nonumber
+2~\p_{\lambda+\delta_0}
\left[{\mathcal C}_n(0)\occupied \Lambda_{n'}^{\rm c}\right]
\\&\le& \nonumber
\left|\p_ \lambda\left[{\mathcal C}_n(0)\occupied (\Lambda_n^{\rm c} \cap \Lambda_{n'})\right]-\p_{\lambda'}
\left[{\mathcal C}_n(0)\occupied (\Lambda_n^{\rm c}\cap \Lambda_{n'})\right]\right|
\\&&\hspace{.5cm} \label{continuity of A}
+~2(2n+1)^d \sup_{x\in \Lambda_n}\p_{\lambda+\delta_0}\left[
x\occupied \Lambda_{n'}^{\rm c}\right].
\end{eqnarray}
We bound the two terms on the right-hand side of
\eqref{continuity of A}.

\noindent{\it (a)}
First we prove that for all $\varepsilon>0$ there exists $n'>n$ such that for
all $x \in \Lambda_n$
\begin{equation}\label{1st claim}
\p_{\lambda+\delta_0}
\left[x\occupied \Lambda_{n'}^{\rm c}\right]<\varepsilon (2n+1)^{-d}/ 4.
\end{equation}
This is done as follows.
For $m>n$ we define the following events
\begin{equation*}
L_m =\{ x\occupied \partial \Lambda_{m+1}\}
=\{ x\occupied (\Lambda_{m+1}\setminus \Lambda_m)\}.
\end{equation*}
This implies for $n'>n$ that
\begin{equation*}
E_{n'}~\stackrel{\rm def.}{=}~
\{x\occupied \Lambda_{n'}^{\rm c}\}=\bigcup_{m\ge n'}L_m.
\end{equation*}
Moreover, note that $E_{n'}$ is decreasing in $n'$, 
\begin{equation*}
\limsup_{n'\to \infty}\p_{\lambda+\delta_0}\left[E_{n'}\right]
=\lim_{n'\to \infty}\p_{\lambda+\delta_0}\left[E_{n'}\right]
=\p_{\lambda+\delta_0}\left[\bigcap_{n'>n}E_{n'}\right]
=\p_{\lambda+\delta_0}\left[\bigcap_{n'>n} \left(\bigcup_{m\ge n'}L_m\right)
\right].
\end{equation*}
We prove \eqref{1st claim}
by contradiction. Assume that \eqref{1st claim} does not
hold true, i.e.~$\limsup_{n'\to \infty}\p\left[E_{n'}\right]>0$.
Then the first lemma of Borel-Cantelli implies
\begin{equation*}
\infty=
\sum_{m>n} \p_{\lambda+\delta_0} \left[L_m\right]
=\sum_{m>n} \p_{\lambda+\delta_0} \left[x\occupied (\Lambda_{m+1}\setminus \Lambda_m)\right]
=\E_{\lambda+\delta_0} \left[
\sum_{m>n} 1_{\{x\occupied (\Lambda_{m+1}\setminus \Lambda_m)\}}\right].
\end{equation*}
The latter implies that the degree distribution $D_x=|\{y\in \Z^d; x\occupied
y \}|$ has an infinite mean. This is a contradiction to Theorem 2.2
of \cite{Remco} saying that for $\min\{\alpha,\beta\alpha\}>d$ 
the survival function of the degree distribution has 
a power-law decay with rate $\alpha\beta/d>1$ which provides a
finite mean. Therefore, claim \eqref{1st claim} holds true.

\noindent
{\it (b)} For all $\varepsilon>0$ and all $n'>n$ 
there exists $\delta_1 \in (0,\delta_0)$ such that for all
$\lambda'\in (\lambda-\delta_1,\lambda+\delta_1)$
\begin{equation}\label{2nd claim}
\left|\p_ \lambda\left[{\mathcal C}_n(0)\occupied 
(\Lambda_n^{\rm c} \cap \Lambda_{n'})\right]-\p_{\lambda'}
\left[{\mathcal C}_n(0)\occupied (\Lambda_n^{\rm c}\cap \Lambda_{n'})\right]\right|
<\varepsilon / 2.
\end{equation}
Note that $\Lambda_{n'}$ only contains finitely many edges
of finite distance. Therefore, continuity in $\lambda$ is straightforward
which provides claim \eqref{2nd claim}.

Combining \eqref{1st claim} and \eqref{2nd claim}
provides continuity of $\p_\lambda[A_n]$ in $\lambda$ for all
$n$, see also \eqref{continuity of A}. Therefore,
right-continuity of $\lambda \mapsto \theta(\lambda, \alpha)$
follows. This finishes the proof of 
Theorem \ref{theorem continuous}.
\end{proof}

\subsection{Proofs of the graph distances}

In this section we prove Theorem \ref{chemical distance}.
Statement (a) of  Theorem \ref{chemical distance} is proved in 
Theorems 5.1 and 5.3 of \cite{Remco},
the lower bound of statement (b1) is proved in  Theorem 5.5 of \cite{Remco}.
Therefore, there remain the proofs of
the upper bound in (b1) and of the lower bound in (b2)
of Theorem \ref{chemical distance}.

The proof of the upper bound in 
Theorem~\ref{chemical distance} (b1) follows from the following proposition and the fact 
that $\alpha \mapsto\Delta(\alpha,2d)=\log2/ \log(2d/\alpha)$ is a continuous function.
The following proposition corresponds to Proposition 4.1 in \cite{biskup} in the
homogeneous long-range percolation model.

\begin{Proposition}\label{chemical distance proposition 1}
Let $\alpha\in(d,2d)$ and $\tau=\beta\alpha/d>2$ and $\lambda>\lambda_c$. 
For each $\Delta'>\Delta=\Delta(\alpha,2d)=\log2/ \log(2d/\alpha)$ and each 
$\varepsilon>0$, there exists $N_0<\infty$ such that
$$\p\left[d(x,y)\geq (\log|x-y|)^{\Delta'},x,y\in \mathcal C\right]\leq \varepsilon$$
holds for all $x,y\in\Z^d$ with $|x-y|\ge N_0$.
\end{Proposition}

\begin{sketch}[of proof of Proposition \ref{chemical distance proposition 1}]
We only sketch the proof because it is almost identical to the one in \cite{biskup}.
Definition 1 and Figure 1 of \cite{biskup} defines for $x,y\in \Z^d$
a hierarchy of depth $m\in \N$ 
connecting $x$ and $y$
as the following collection of vertices:
\begin{equation*}
{\cal H}_m(x,y)=\left\{ 
z_\sigma \in \Z^d;~ \sigma\in \{0,1\}^k \text{ for } k=1,\ldots, m \right\},
\end{equation*}
is a hierarchy of depth $m \in \N$ connecting $x$ and $y$ if
\begin{itemize}
\item[(1)] $z_0=x$ and $z_1=y$,
\item[(2)] $z_{\sigma00}=z_{\sigma0}$ and $z_{\sigma11}=z_{\sigma1}$ for all $k=0,\ldots, m-2$
and $\sigma \in \{0,1\}^k$,
\item[(3)] for all $k=0,\ldots, m-2$ and and $\sigma \in \{0,1\}^k$ such that 
$z_{\sigma01}\neq z_{\sigma10}$ the edge between 
$z_{\sigma01}$ and $z_{\sigma10}$ is occupied,
\item[(4)] each bond $(z_{\sigma01},z_{\sigma10})$ specified in (3) appears only once
in ${\cal H}_k(x,y)$.
\end{itemize}
The pairs of vertices $(z_{\sigma00},z_{\sigma01})$ and $(z_{\sigma10},z_{\sigma11})$
are called gaps. The proof is then based on the fact that for large distances
$|x-y|$ the event ${\cal B}_m$ of the existence of a hierarchy ${\cal H}_m(x,y)$ of depth $m$ 
that connects $x$ and $y$ 
through points $z_\sigma$ which are dense is very likely ($m$ appropriately chosen),
see Lemma 4.3 in \cite{biskup}, in particular formula (4.18) in \cite{biskup}
(where the key is Corollary \ref{corollary big clusters} $(ii)$).
On this likely event ${\cal B}_m$, Lemma 4.2 of \cite{biskup} then proves that
the graph distance cannot be too large, see (4.8) in \cite{biskup}.
We can now almost literally translate Lemmas 4.2 and 4.3 of \cite{biskup} to
our situation. The only changes are that in formulas (4.16) and (4.21) of \cite{biskup}
we need to replace $\beta>0$ of \cite{biskup}'s notation by $\lambda$ in our notation
and we need to use that the weights $W_x$ are at least one, a.s. We refrain from giving more
details.
\end{sketch}

There remains the proof of the lower bound in (b2)
of Theorem \ref{chemical distance}. We use a renormalization technique
which is based on a scheme introduced by \cite{Berger2}. 
Choose an integer valued sequence $a_n>1$, $n\in\N_0$, 
and define the box lengths $(m_n)_{n\in\N_0}$ as follows: 
set $m_0=a_0$ and for $n\in\N$, 
\begin{equation*}
m_n=a_nm_{n-1}=m_0\prod_{i=1}^na_i=\prod_{i=0}^na_i.
\end{equation*}
Define the $n$-stage boxes, $n\in\N_0$, by 
\begin{equation*}
B_{m_n}(x)=x + [0,m_n-1]^d, \qquad \text{ for $x\in \Z^d$.}
\end{equation*}
For $n\ge 1$, the children of  $n$-stage box $B_{m_n}(x)$ are the
$a_n^d$ disjoint $(n-1)$-stage boxes
\begin{equation*}
B_{m_{n-1}}(x+ym_{n-1})=x+ym_{n-1}+[0,m_{n-1}-1]^d\subset \Z^d
\qquad \text{ with $y\in ([0,a_n-1]^d\cap \Z^d)$.}
\end{equation*}
We are going to define
{\it good} $n$-stage boxes $B_{m_n}(\cdot)$, 
note that we need a different definition from Definition 2 of \cite{Berger2}.

\begin{Definition}[good $n$-stage boxes]\label{definition good boxes}
Choose $n\in \N_0$ and $x \in \Z^d$ fixed.
\begin{itemize}
\item $0$-stage box $B_{m_0}(x)$ is {\it good} under a given
edge configuration if there is no 
occupied edge in $B_{m_0}(x)$ with size larger than
$m_0/100$.
\item $n$-stage box $B_{m_n}(x)$, $n\ge1$, is {\it good} under a given
edge configuration if for all $j\in \{-1,0,1\}^d$
\begin{itemize}
\item[(a)] there is no occupied edge in $B_{m_n}\left(x+j\frac{m_{n-1}}{2}\right)$ with size larger than
$m_{n-1}/100$; and
\item[(b)] among the children of $B_{m_n}\left(x+j\frac{m_{n-1}}{2}\right)$ there
are at most $3^d$ that are not good.
\end{itemize}
\end{itemize}
\end{Definition}

\begin{Lemma}\label{good children 0}
Assume $\min\{\alpha, \beta\alpha\} > d$. 
For all $\delta \in (0,\alpha(\beta \wedge 1)-d)$ 
there exist $t_0 \ge 1$ and a constant $c_1>0$ such that for all $t \ge t_0$ and all $s\ge 1$,
\begin{equation*}
\p\left[\text{there is an occupied edge in $[0,s-1]^d$ with size larger than $t$} \right] 
\le 
c_1 s^d t^{d-\alpha( \beta \wedge 1)+\delta}.
\end{equation*}
\end{Lemma}

\begin{proof}[of Lemma \ref{good children 0}]
Let $W_1$ and $W_2$ be two independent 
random variables  each 
having a Pareto distribution with parameters $\theta=1$ and 
$\beta>0$. 
For $u \ge 1$ we have, using integration by parts in the first step, 
\begin{eqnarray*}
\E\left[ \frac{ W_1W_2}{u} \wedge 1 \right]
&=&
\frac{1}{u}
+
\frac{1}{u}  \int_1^u \p[W_1W_2> v]dv
~=~
\frac{1}{u} 
+ \frac{1}{u} \int_1^u v^{-\beta}(1+ \beta \log v) dv
\\&\le&
(1+ \beta \log u)\left( u^{-(\beta \wedge 1)} 
+ \frac{1}{u} \int_1^u v^{-\beta}dv\right)
\\&\le&
\max\{1+\log u,1+1_{\{\beta\neq1\}}/|\beta-1|\} \left( 1+ \beta \log u\right ) u^{-(\beta \wedge 1)},
\end{eqnarray*}
where the last step follows by distinguishing between the cases 
$\beta=1$, $\beta > 1$ and $\beta < 1$. 
Choose $t_0$ so large that $\lambda^{-1}t_0^\alpha \ge 1$ which, 
together with the above calculations, 
implies that for all $t\ge t_0$ and 
$x,y\in\Z^d$ with $|x-y|>t\ge t_0$, 
\begin{eqnarray*}
\E\left[\frac{\lambda W_xW_y}{|x-y|^\alpha} \wedge 1\right]
&\le& 
(1+1_{\{\beta\neq1\}}/|\beta-1|)\left( 1+ \max\{1, \beta\} \log (\lambda^{-1}|x-y|^\alpha)\right )^2 \left(\lambda^{-1}|x-y|^\alpha\right)^{-(\beta \wedge 1)}
\\&\le&
|x-y|^{-\alpha(\beta \wedge 1)+\delta},
\end{eqnarray*}
where the second inequality holds for all $|x-y|>t\ge t_0$ with 
$t_0$ large enough. 
It follows that for all $t\ge t_0$, 
using $1-e^{-x}\le x\wedge1$, 
\begin{eqnarray*}
&&\hspace{-0.5cm}
\p\left[\text{there is an occupied edge  in $[0,s-1]^d$ with size larger than $t$} \right] 
~\le~
\sum_{\substack{x,y \in  [0,s-1]^d: \\ |x-y| > t}} 
\E\left[\frac{\lambda W_xW_y}{|x-y|^\alpha} \wedge 1\right]
\\&&\le 
\sum_{\substack{x,y \in  [0,s-1]^d: \\ |x-y| > t}} 
|x-y|^{-\alpha(\beta \wedge 1)+\delta}
~\le~
s^d\sum_{y \in  \Z^d: \; |y| > t} |y|^{-\alpha(\beta \wedge 1)+\delta}.
\end{eqnarray*}
Hence, for an appropriate constant $c_1>0$ and for all $t\ge t_0$ 
with $t_0$ sufficiently large, 
\begin{equation*}
\p\left[\text{there is an occupied edge  in $[0,s-1]^d$ with size larger than $t$} \right]
~\le~
c_1 s^d t^{d-\alpha( \beta \wedge 1)+\delta},
\end{equation*}
which finishes the proof of Lemma \ref{good children 0}. 
\end{proof}

\begin{Lemma}\label{good children}
Assume $\min \{ \alpha, \beta \alpha \} > 2d$. 
For $a_n=n^2$, $n\ge1$, and $a_0$ sufficiently large we have
\begin{equation*}
\sum_{n\ge 0} \p \left[B_{m_n}(0) \text{ is not good }\right] <\infty.
\end{equation*}
\end{Lemma}
\noindent
This lemma is the analog in our model to Lemma 1 of \cite{Berger2}
and provides a Borel-Cantelli type of result that eventually the 
boxes $B_{m_n}(0)$ are good, a.s., for all $n$ sufficiently large.

\begin{proof}[of Lemma \ref{good children}]
We prove by induction that 
$\psi_n=\p\left[\text{$B_{m_n}(0)$ is not good} \right]$ 
is summable. 
Choose $\delta \in (0,\alpha(\beta\wedge1)-2d)$ and set 
$\gamma=\min\{\alpha,\beta\alpha\}-2d-\delta>0$. 
For $m_0$ sufficiently large we obtain 
by Lemma \ref{good children 0}, 
\begin{eqnarray}
\psi_0
&=&\notag
\p\left[
\text{there is an occupied edge  in $B_{m_0}(0)$ with size larger than $m_0/100$}
\right]
\\&\le&\label{Equation: psi_0}
c_1m_0^d\left(\frac{m_0}{100}\right)^{d-\alpha(\beta\wedge1)+\delta}
~<~
3^{-d}2^{-4d-1}e^{-2},
\end{eqnarray}
where the last step holds true for $m_0$ sufficiently large. 
Because $B_{m_1}(0)$ has only one child (because $a_1=1$) we get 
for $m_0$ sufficiently large 
\begin{equation}\label{Equation: psi_1}
\psi_1
~\le~
3^d\psi_0
~\le~
c_13^dm_0^d\left(\frac{m_0}{100}\right)^{d-\alpha(\beta\wedge1)+\delta}
~<~
3^{-d}2^{-8d-1}e^{-4}.
\end{equation}
For the induction step we note that $n$-stage box $B_{m_n}(0)$ is 
not good if at least one of the $3^d$ translations 
$B_{m_n}(0+j\frac{m_{n-1}}{2})$, $j\in\{-1,0,1\}^d$, fails to have 
property (a) or (b) of Definition \ref{definition good boxes}. 
Using translation invariance and Lemma \ref{good children 0} 
we get  for all $n\ge2$ 
and for all $m_0$ sufficiently large, 
set $c_2= c_1100^{\alpha(\beta\wedge1)-d-\delta}$,
\begin{equation*}
\psi_n
\le 3^d
\left(
c_2a_n^{\alpha(\beta\wedge1)-d-\delta}m_n^{-\gamma}
+
\p\left[
\text{there are at least $3^d+1$ children of $B_{m_n}(0)$ that are not good}
\right]
\right).
\end{equation*}
Note that the event in the probability above ensures that there are at least 
two children $B_{m_{n-1}}(y)$ and $B_{m_{n-1}}(z)$ of 
$B_{m_n}(0)$ that are not good and are separated by at least 
Euclidean distance $2m_{n-1}$. Therefore, using $m_i=a_0(i!)^2$, $i\ge0$, 
the two boxes $B_{m_{n-1}}(y)$ and $B_{m_{n-1}}(z)$ are well separated in the 
sense that the events $\{\!\text{$B_{m_{n-1}}\!(y)\!$ is not\! good}\}$ 
and $\{\text{$B_{m_{n-1}}(z)$ is not good}\}$ are independent. 
Note that for the latter we need to make sure that 
$B_{m_{n-1}}(y+jm_{n-2}/2)$ and $B_{m_{n-1}}(z+lm_{n-2}/2)$ 
are disjoint for all $j,l \in \{-1,0,1\}^d$, which is the case because 
$B_{m_{n-1}}(y)$ and $B_{m_{n-1}}(z)$ have at least 
distance $2m_{n-1}$. 
The independence implies the following bound 
\begin{eqnarray*}
\psi_n
&\le&3^d
\left(
c_2a_n^{\alpha(\beta\wedge1)-d-\delta}m_n^{-\gamma}
+
\binom{a_n^d}{2} \psi_{n-1}^2
\right)
~\le~3^d
\left(
c_2a_n^{\alpha(\beta\wedge1)-d-\delta}m_n^{-\gamma}
+a_n^{2d} \psi_{n-1}^2
\right)
\\&=&3^d
\left(
c_2n^{2(\gamma+d)}\left(m_0 (n!)^2\right)^{-\gamma}
+
n^{4d} \psi_{n-1}^2
\right)
~=~
c_23^d m_0^{-\gamma}n^{2(\gamma+d)}
(n!)^{-2\gamma}
+
3^d n^{4d} \psi_{n-1}^2.
\end{eqnarray*}
It follows that there is $n_0<\infty$ such that for all for all $n\ge n_0$ and $m_0$ large enough
\begin{equation}\label{Equation: n_0}
\psi_n 
~\le~
3^{-d}2^{-4d-2}e^{-2}(n+1)^{-4d}e^{-2n}
+
3^d n^{4d} \psi_{n-1}^2,
\end{equation}
and we can choose $m_0$ so large 
that \eqref{Equation: n_0} holds true also for all $2\le n < n_0$. 
We claim that for all $a_0=m_0$ sufficiently large and all $n\ge0$, 
\begin{equation}\label{Equation: induction}
\psi_n 
~\le~
3^{-d}2^{-4d-1}e^{-2}(n+1)^{-4d}e^{-2n},
\end{equation}
which will imply Lemma \ref{good children} because 
the right-hand side is summable. 
Indeed, \eqref{Equation: induction} is true for $n\in\{0,1\}$ by \eqref{Equation: psi_0} 
and \eqref{Equation: psi_1}. 
Assuming that \eqref{Equation: induction} holds for all $n-1$ 
with $n\ge2$ we get, using \eqref{Equation: n_0}, 
\begin{eqnarray*}
\psi_n
&\le&
3^{-d}2^{-4d-2}e^{-2}(n+1)^{-4d}e^{-2n}
+
3^d n^{4d} \psi_{n-1}^2 
\\&\le&
3^{-d}2^{-4d-2}e^{-2}(n+1)^{-4d}e^{-2n}
+
3^{-d} n^{-4d}2^{-8d-2}e^{-4}e^{-4n+4}
\\&=&
3^{-d}2^{-4d-1}e^{-2}(n+1)^{-4d}e^{-2n}
\left(2^{-1}+
\left(\frac{n+1}{n}\right)^{4d}
2^{-4d-1}e^{-2n+2}
\right)
\\&\le&
3^{-d}2^{-4d-1}e^{-2}(n+1)^{-4d}e^{-2n}
\left(2^{-1}+
2^{-1}
\right),
\end{eqnarray*}
where the last step follows since $(n+1)/n\le2$ 
and $e^{-2n+2}\le1$. 
\end{proof}

The following lemma is the analog of Proposition 3 of \cite{Berger2} and it depends on 
Lemma 2 of \cite{Berger2} and Lemma~\ref{good children}.
Since its proof is completely similar to the one of Proposition 3 of \cite{Berger2}
once Lemma \ref{good children} has been established we skip this proof.
\begin{Lemma}[Proposition 3 of \cite{Berger2}]
\label{minimal length of paths 2}
Choose   $a_n=n^2$ for $n\ge 1$.
There exists a constant $c_3>0$ such that for
every $n$ sufficiently large,
if  for every $j\in \{-1,0,1\}^d$ the $n$-stage box
$B_{m_n}\left(0+j\frac{m_{n}}{2}\right)$
is good and for every $l>n$ the $l$-stage boxes $\widehat{B}_{m_l}$ centered at 
$B_{m_n}(0)$ are good,
then if $x,y\in B_{m_n}(0)$ satisfy $|x-y|>m_n/8$ then 
$d(x,y)\ge c_3|x-y|$.
\end{Lemma}

\begin{proof}[of Theorem \ref{chemical distance} (b2)]
Lemma \ref{good children} says that, a.s., the $l$-stage boxes 
$\widehat{B}_{m_l}$ are eventually good for all $l\ge n$.
Moreover, from Lemma \ref{minimal length of paths 2} we obtain the linearity
in the distance for these good boxes which says that, a.s., for $n$ sufficiently large and
$|x|>m_n/8$ we have $d(0,x)\ge c_3|x|$.
\end{proof}


~

{\bf Acknowledgment.} We thank Noam Berger and Remco van der Hofstad
for helpful discussions and comments. 
R.~S.~H.~acknowledges support from Swiss National Science Foundation, grant 132909.

\bibliographystyle{abbrv}
\bibliography{References}

\end{document}